\documentclass[11pt,reqno]{amsart}
\usepackage[T1]{fontenc}
\usepackage[utf8]{inputenc}
\usepackage[english]{babel}
\usepackage{amssymb,amsthm,amsmath}
\usepackage{booktabs}
\usepackage[margin=1in]{geometry}
\usepackage[colorlinks=true,linkcolor=blue,citecolor=blue,urlcolor=blue]{hyperref}

\newtheorem{theorem}{Theorem}[section]
\newtheorem{proposition}[theorem]{Proposition}
\newtheorem{lemma}[theorem]{Lemma}
\newtheorem{corollary}[theorem]{Corollary}
\theoremstyle{remark}\newtheorem*{remark}{Remark}
\numberwithin{equation}{section}
\let\a\alpha \let\b\beta  \let\g\gamma
\newcommand{\He}{\mathrm{He}}\newcommand{\E}{\mathbb E}\newcommand{\nab}{\nabla}
\newcommand{\R}{\mathbb R}\newcommand{\dd}{\operatorname{d}}\newcommand{\sgn}{\operatorname{sgn}}

\title[Appell Polynomials in Shifted Asymptotic Expansions]
{Appell Polynomials in Shifted Asymptotic Expansions:
the Mills ratio, Hermite polynomials,
and Stieltjes bounds}

\author{Tomislav~Buri\'c}
\address{University of Zagreb, Faculty of Electrical Engineering and Computing,
Unska 3, 10000 Zagreb, Croatia}
\email{tomislav.buric@fer.hr}
\thanks{Corresponding author: Tomislav Buri\'c.}

\author{Neven~Elezovi\'c}
\address{University of Zagreb, Faculty of Electrical Engineering and Computing,
Unska 3, 10000 Zagreb, Croatia}
\email{neven@element.hr}

\author{Lenka~Mihokovi\'c}
\address{University of Zagreb, Faculty of Electrical Engineering and Computing,
Unska 3, 10000 Zagreb, Croatia}
\email{lenka.mihokovic@fer.hr}


\date{\today}
\begin{document}

\begin{abstract}
A theorem of Buri\'c, Elezovi\'c and Vuk\v si\'c states that translating the argument
in an asymptotic expansion,
$f(x)\sim\sum(-1)^na_nx^{-n-1}$,  to $f(x+t)$, replaces the constant coefficients
$a_n$ by the Appell polynomials $R_n(t)$ generated by $(a_n)$. Their motivating
examples came from the gamma and polygamma functions, where Bernoulli polynomials
occur. We extend this construction beyond that setting and identify the
Borel--Laplace representation
$L_A(x)=\int_0^\infty e^{-xs}A(-s)\,\dd s$, whenever the integral exists. For the
Gaussian kernel $A(-s)=e^{-s^2/2}$, this representation gives the
\emph{Mills--Hermite} expansion
$M(x+t)\sim\sum(-1)^n\He_n(t)x^{-n-1}$, extending the classical scalar expansion.
We establish an explicit remainder estimate and derive finite-difference formulas
that remain in the Hermite polynomial algebra. As an application, we obtain a
large-threshold expansion of the non-central Gaussian tail with explicit polynomial
dependence on the mean. If $A(-s)$ is itself the Laplace transform of a positive
measure, then the coefficients form a Stieltjes moment sequence. The associated
Pad\'e convergents give a systematic hierarchy of two-sided bounds in which
successive lower and upper approximants incorporate the moments one at a time.
\end{abstract}

\maketitle

\section{Introduction}
In a sequence of papers on the asymptotics of the gamma function and its relatives
\cite{BE034,BEV057}, the
Bernoulli polynomials and, more precisely, divided differences of Bernoulli
polynomials, emerged as the natural coefficients. A striking instance is the digamma
function: from $\psi(x)\sim\log x-\frac1{2x}-\sum_{k\ge1}\frac{B_{2k}}{2k}x^{-2k}$ one
has, for every real $t$, the \emph{shifted} expansion
\[
\psi(x+t)\sim\log x+\sum_{n\ge1}(-1)^{n+1}\frac{B_n(t)}{n}\,x^{-n},
\]
with the Bernoulli \emph{polynomials} $B_n(t)$ in place of the numbers $B_n$. Buri\'c,
Elezovi\'c and Vuk\v si\'c \cite{BEV057} explained this as an instance of a general
principle: a shift of the argument always converts the coefficients of an asymptotic
series into Appell polynomials. We recall their result in \S\ref{sec:framework}.

Those papers analysed gamma-related functions, for which the relevant Appell
sequence is the Bernoulli sequence. Euler polynomials also lead to elementary
combinations of digamma functions and are not considered here. We study two
different classes in which the same construction gives further results. The
exponential generating function $A(z)=\sum a_nz^n/n!$ provides the Borel kernel
$A(-s)$ of the alternating asymptotic series. Whenever the corresponding Laplace
integral exists, one has
\[
 f(x)=L_A(x):=\int_0^\infty e^{-xs}A(-s)\,\dd s
 \sim\sum_{n\ge0}(-1)^n\frac{a_n}{x^{n+1}}.
\]
We consider two classes of kernels:
\begin{itemize}
\item the \emph{Gaussian} kernel $A(-s)=e^{-s^2/2}$, whose Appell sequence is Hermite and
whose Laplace transform is the Mills ratio --- the reciprocal of the Gaussian hazard
(\S\ref{sec:mills}--\ref{sec:apps});
\item the \emph{Stieltjes} kernels, $A(-s)=\int e^{-\sigma s}\,\dd\mu(\sigma)$ for a
positive measure $\mu$, whose coefficients are the moments of $\mu$ and whose Laplace
transform is a Stieltjes function (\S\ref{sec:stieltjes}).
\end{itemize}
For the Gaussian kernel, our main result is the shifted Mills expansion
\eqref{eq:millsher}. It contains the full Hermite polynomials $\He_n(t)$ and includes
an explicit remainder estimate. The shift parameter can then be treated
algebraically, which gives finite-difference formulas and a large-threshold expansion
of the non-central Gaussian tail. For Stieltjes kernels, positivity of $\mu$ makes
$f$ a Stieltjes function. Its Pad\'e convergents provide lower and upper bounds
determined by successive moments of the generating sequence.

The shift may also depend on the asymptotic variable. The proof of
Theorem~\ref{thm:shift} is uniform for $t$ in compact sets and therefore applies to
every bounded family $t_x$, even if $t_x$ has no limit. This is formulated in
Corollary~\ref{cor:moving}. Such shifts occur naturally in De~Moivre's formula for
the mean absolute deviation of a binomial variable. The formula involves the lattice
point $\lceil Np\rceil$, whose displacement from the mean $Np$ is bounded but need
not converge. The resulting expansion is expressed, at every order, by Bernoulli
polynomials evaluated at this moving shift \cite{ElezovicMAD}.

\paragraph{Relation with earlier work.} The algebra of Appell sequences is part of the umbral
calculus \cite{Roman}, while the passage from a local expansion of a Laplace kernel to
an asymptotic expansion is Watson's lemma \cite{Olver}. L\'opez and Temme
\cite{LopezTemme} use the
same structural idea, with Hermite polynomials as coefficients, for generalized
Bernoulli, Euler, Bessel and Buchholz polynomials. The Mills ratio is not among their examples. The connection between the Mills ratio and
the Hermite polynomials is itself classical --- it underlies Jacobi's representation
$M^{(n)}=P_nM-Q_n$ of the derivatives --- and Luci\'c \cite{Lucic} starts from the
Laplace representation of Lemma~\ref{lem:lap} and, combining it with the Hermite
expansion of monomials, recovers those derivative polynomials directly. The scalar
asymptotic series and its Laplace continued fraction are likewise standard
\cite{DLMF,GasullUtzet}. The contribution of the present paper is to place the shift
theorem in a Borel--Laplace framework and to develop its consequences for the Mills
ratio and for Stieltjes moment sequences. In the Gaussian case we obtain
\eqref{eq:millsher}, with the full polynomials $\He_n(t)$ and an explicit remainder
estimate, and use it for finite differences and the non-central tail expansion. In
the Stieltjes case we organise the classical Pad\'e bracketing
\cite{Baker,Akhiezer,Wall} directly in terms of the Appell-generating moment sequence.
Thus, the same shift mechanism gives a common treatment of two classes that are
usually considered separately.

\section{The shift theorem and finite differences}\label{sec:framework}
A sequence of polynomials $(R_n)$ is an \emph{Appell sequence} if there is a formal
power series such that
\[
A(z)=\sum_{k\ge0}a_k\frac{z^k}{k!},\qquad a_0\ne0,\qquad
A(z)e^{tz}=\sum_{n\ge0}R_n(t)\frac{z^n}{n!}.
\]
Then
\begin{equation}\label{eq:appell}
R_n(t)=\sum_{k=0}^n\binom nk a_k\,t^{n-k},\qquad R_n'(t)=nR_{n-1}(t),\qquad R_n(0)=a_n,
\end{equation}
and the differential recurrence, together with $R_n(0)=a_n$, characterises the
sequence. The Bernoulli, probabilists' Hermite and Euler polynomials correspond,
respectively, to
\[
A(z)=\frac{z}{e^z-1},\qquad A(z)=e^{-z^2/2},\qquad
A(z)=\frac{2}{e^z+1}.
\]

\begin{theorem}[\cite{BEV057}]\label{thm:shift}
If
\[
f(x)\sim\sum_{n\ge0}(-1)^n\frac{a_n}{x^{n+1}}
\qquad (x\to\infty),
\]
then, for every fixed $t\in\R$,
\begin{equation}\label{eq:shift}
f(x+t)\sim\sum_{n\ge0}(-1)^n\frac{R_n(t)}{x^{n+1}},
\end{equation}
where $(R_n)$ is the Appell sequence generated by $(a_n)$.
\end{theorem}
\begin{proof}
Fix $N$. By the definition of the asymptotic expansion, for fixed $t$,
\[
f(x+t)=\sum_{k=0}^{N}(-1)^k a_k(x+t)^{-k-1}+o\bigl(x^{-N-1}\bigr)\qquad(x\to\infty),
\]
since $(x+t)^{-N-1}=x^{-N-1}\bigl(1+o(1)\bigr)$ for fixed $t$. Each explicit term has the
finite Taylor expansion
\[
(x+t)^{-k-1}=\sum_{j=0}^{N-k}(-1)^j\binom{j+k}{j}t^j x^{-k-j-1}+O\bigl(x^{-N-2}\bigr).
\]
Substituting and collecting the coefficient of $x^{-n-1}$ for $n\le N$ gives
\[
(-1)^n\sum_{k=0}^n\binom nk a_kt^{n-k}=(-1)^nR_n(t).
\]
The finitely many
$O(x^{-N-2})$ tails and the $o(x^{-N-1})$ remainder combine to $o(x^{-N-1})$. As $N$ is
arbitrary, this proves \eqref{eq:shift}.
\end{proof}

\begin{remark}
The formal series on the right of \eqref{eq:shift} satisfies
$\partial_xF=\partial_tF$, because $R_n'=nR_{n-1}$. This is a useful consistency
test, but it is not a converse for asymptotic expansions: an exponentially small
(beyond-all-orders) term depending separately on $x$ and $t$ is invisible to every
algebraic order.
Only if an actual $C^1$ function on a connected domain satisfies
$\partial_xF=\partial_tF$ exactly can one conclude that it is a function of $x+t$.
\end{remark}

\subsection{A moving shift and De Moivre's formula}\label{sec:moving}
Theorem~\ref{thm:shift} is stated for a fixed $t$. Since the shift enters the proof
only through bounds that are uniform on compact sets, the same argument also covers
bounded shifts that depend on $x$.

\begin{corollary}[Moving shifts]\label{cor:moving}
The expansion \eqref{eq:shift} holds uniformly for $t$ in compact subsets of $\R$: for every
$T>0$ and every $N$,
\[
f(x+t)=\sum_{n=0}^{N}(-1)^n\frac{R_n(t)}{x^{n+1}}+o\bigl(x^{-N-1}\bigr)
\qquad(x\to\infty),
\]
the $o$ being uniform for $|t|\le T$. Consequently, if $(t_x)_{x>0}$ is \emph{any} family with
$|t_x|\le T$, then
\begin{equation}\label{eq:moving}
f(x+t_x)=\sum_{n=0}^{N}(-1)^n\frac{R_n(t_x)}{x^{n+1}}+o\bigl(x^{-N-1}\bigr),
\end{equation}
even when $t_x$ has no limit as $x\to\infty$.
\end{corollary}
\begin{proof}
Both error terms in the proof of Theorem~\ref{thm:shift} are uniform for $|t|\le T$. The
expansion of $f$ is used at $y=x+t\ge x-T$, which tends to infinity uniformly; and in the Taylor
expansion of $(x+t)^{-k-1}$ the remainder $O(x^{-N-2})$ has a constant depending only on $N$ and
$T$, for $x\ge2T$. Evaluating the uniform statement along the family gives \eqref{eq:moving}.
\end{proof}

If $t_x$ oscillates, the coefficients $R_n(t_x)$ need not converge, but the
expansion remains valid to every order because the error is uniform.

The following example is based on one of the earliest formulas in probability.
Let $X\sim\mathrm{Bin}(N,p)$, $q=1-p$. De~Moivre's closed form for the mean
absolute deviation,
\[
\E\,|X-Np|=2\nu q\binom N\nu p^{\nu}q^{N-\nu},\qquad \nu=\lceil Np\rceil,
\]
evaluates the binomial mass at the first lattice point \emph{above} the mean. The displacement
\[
h_N=\nu-Np=\lceil Np\rceil-Np\in[0,1)
\]
is bounded but does not converge --- for irrational $p$ it is equidistributed in $[0,1)$ --- and it
is exactly a moving shift in the sense of \eqref{eq:moving}. Writing the binomial mass as a
quotient of gamma functions and applying the shifted Stirling series, that is, the log-gamma
instance of the shift principle with the Bernoulli sequence, one obtains \cite{ElezovicMAD}
\begin{equation}\label{eq:mad}
\begin{aligned}
\E\,|X-Np|
&\sim\sqrt{\frac{2Npq}{\pi}}\,
\exp\Bigl(\sum_{k\ge1}\frac{a_k(h_N;p)}{N^{k}}\Bigr),\\
a_k(h;p)
&=\frac{(-1)^{k+1}}{k(k+1)}
\Bigl[B_{k+1}\\
&\hspace{25mm}
-\bigl(p^{-k}+(-1)^{k+1}q^{-k}\bigr)B_{k+1}(h)\Bigr].
\end{aligned}
\end{equation}
Every coefficient is a Bernoulli polynomial at the moving shift; no additional
elementary term occurs.

This property is related to the Appell structure of the Bernoulli sequence. Among
all Appell sequences, the Bernoulli sequence is characterised by the difference
identity
\begin{equation}\label{eq:charB}
\begin{aligned}
R_n(t+1)-R_n(t)=n\,t^{n-1}\quad(n\ge1)
&\iff A(z)\bigl(e^{z}-1\bigr)=z\\
&\iff A(z)=\frac{z}{e^{z}-1},
\end{aligned}
\end{equation}
as one reads off $A(z)e^{tz}(e^z-1)=z\,e^{tz}$; in the same way the Euler sequence is the unique
Appell sequence with $R_n(t+1)+R_n(t)=2t^n$, that is, $A(z)(e^z+1)=2$. In the binomial
computation the three gamma factors carry the shifts $1$, $h+1$ and $1-h$, and it is
\eqref{eq:charB}, together with the reflection $B_m(1-t)=(-1)^mB_m(t)$, that reduces them to the
single polynomial $B_{k+1}(h)$, together with an elementary term in $h^{k}$. This
term is cancelled, for every $k$, by the logarithm of the De~Moivre prefactor
$\nu=Np+h$. The result is the Appell series \eqref{eq:mad}. It remains to determine
whether this cancellation follows from a general mechanism for Appell sequences
satisfying an identity of the type \eqref{eq:charB}.

\subsection{Borel--Laplace realisation}
The series $A(z)$ is an algebraic datum; it need not converge, and even when it
converges locally the associated Laplace integral need not exist. When it does,
the signs in the asymptotic expansion determine which kernel must be used.

\begin{proposition}[Laplace--Appell correspondence]\label{prop:laplace-appell}
Suppose the Borel kernel $A(-s)$ is defined and infinitely differentiable on
$[0,\infty)$, satisfies $|A(-s)|\le Ce^{bs}$ for some constants $C>0$, $b\ge0$, and has
the asymptotic Taylor behaviour
\[
A(-s)\sim\sum_{n\ge0}(-1)^na_n\frac{s^n}{n!}\qquad(s\to0^+).
\]
Then for every fixed $t\in\R$ the integral
\[
L_A(x+t):=\int_0^\infty e^{-xs}A(-s)e^{-ts}\,\dd s
\]
converges absolutely for $x>b-t$, and
\begin{equation}\label{eq:laplace}
L_A(x+t)\sim\sum_{n\ge0}(-1)^n\frac{R_n(t)}{x^{n+1}}\qquad(x\to\infty).
\end{equation}
The growth bound is only a sufficient condition; for a rapidly decaying kernel such as
the Gaussian $A(-s)=e^{-s^2/2}$ the integral converges for all real $x$.
\end{proposition}
\begin{proof}
Multiplying the asymptotic expansion of $A(-s)$ at $s=0$ by $e^{-ts}$ and collecting
powers gives, by the Appell addition formula,
\[
A(-s)e^{-ts}\sim\sum_{n\ge0}(-1)^nR_n(t)\frac{s^n}{n!}\qquad(s\to0^+),
\]
while the bound $|A(-s)e^{-ts}|\le Ce^{(b-t)s}$ controls the tail. Watson's lemma,
together with $\int_0^\infty e^{-xs}s^n\,\dd s=n!/x^{n+1}$, gives \eqref{eq:laplace}.
\end{proof}

\subsection{Finite differences in the shift}
The polynomial structure of the Appell coefficients $R_n$ is useful for
\emph{finite differences}, in which $f$ is evaluated at points whose separation does
not scale with $x$. These differences are used below for discrete convexity and for
increments in the shift. Differentiation in the shift gives no additional operation:
since
$\partial_tf(x+t)=\partial_xf(x+t)=f'(x+t)$, a $t$-derivative only reproduces an
$x$-derivative. We centre the samples at $x+\a$ and use the half-spread $\delta$.
The symmetric specialisation of the intrinsic
variables used in \cite{BEV057} is
\begin{equation}\label{eq:albe}
\a\quad(\text{midpoint}),\qquad \b=-\delta^2\quad(\text{spread}).
\end{equation}
The dependence is even in $\delta$, hence polynomial in $\b$; for fixed $\b$ the
coefficients satisfy the Appell lowering relation in the midpoint $\a$.

\begin{proposition}[First and second symmetric differences]\label{prop:diff}
Fix $\a\in\R$ and $\delta\ge0$, both independent of $x$, and put $\b=-\delta^2\le0$. For
the symmetric pair $x+\a\pm\delta$ define
\[
\nab_m(\a,\b):=
\frac{R_{m+1}(\a+\delta)-R_{m+1}(\a-\delta)}
     {2(m+1)\delta},
\]
with the value at $\delta=0$ understood by continuity. Then
\begin{equation}\label{eq:dd}
\frac{f(x+\a+\delta)-f(x+\a-\delta)}{2\delta}\sim
 -\sum_{m\ge0}(-1)^m\frac{(m+1)\,\nab_m(\a,\b)}{x^{m+2}},
\end{equation}
where
\begin{equation}\label{eq:nabla}
\nab_m(\a,\b)=\sum_{i=0}^{\lfloor m/2\rfloor}
\frac1{2i+1}\binom m{2i}(-\b)^iR_{m-2i}(\a).
\end{equation}
The symmetric second difference has the expansion
\begin{equation}\label{eq:diff2}
f(x+\a+\delta)-2f(x+\a)+f(x+\a-\delta)
\sim\sum_{n\ge2}(-1)^n\frac{\Delta_n(\a,\b)}{x^{n+1}},
\end{equation}
with
\begin{equation}\label{eq:Delta}
\Delta_n(\a,\b)=2\sum_{i=1}^{\lfloor n/2\rfloor}
\binom n{2i}(-\b)^iR_{n-2i}(\a).
\end{equation}
Here $\nab_m$ is defined for $m\ge0$ and $\Delta_n$ for $n\ge0$, with
$\Delta_0=\Delta_1=0$ (empty sums). For fixed $\b$, both satisfy the Appell lowering
relation
\[
\partial_\a\nab_m=m\nab_{m-1}\ \ (m\ge1),\qquad
\partial_\a\Delta_n=n\Delta_{n-1}\ \ (n\ge1).
\]
\end{proposition}
\begin{proof}
Taylor's formula and the Appell identity give
\[
R_{m+1}^{(j)}(\a)=\frac{(m+1)!}{(m+1-j)!}R_{m+1-j}(\a).
\]
Subtracting the expansions at $\a+\delta$ and $\a-\delta$ cancels the even
powers of $\delta$; division by $2(m+1)\delta$ gives \eqref{eq:nabla}.
Adding those two expansions and subtracting $2R_n(\a)$ cancels the odd powers
and gives \eqref{eq:Delta}. Substituting these coefficients into \eqref{eq:shift} and
matching the coefficient of $x^{-m-2}$ (respectively $x^{-n-1}$) proves
\eqref{eq:dd} and \eqref{eq:diff2}. Differentiation with respect to $\a$
proves the last two identities.
\end{proof}

For example,
\[
\nab_0=a_0,\quad \nab_1=R_1(\a),\quad
\nab_2=R_2(\a)-\tfrac13a_0\b,\quad
\nab_3=R_3(\a)-\b R_1(\a).
\]
As $\b\to0$, $\nab_m\to R_m(\a)$ and \eqref{eq:dd} becomes the expansion of
$f'(x+\a)$. The ordinary increment is obtained from
$\a=\delta=t/2$; its $a_0/x$ term cancels, so
$f(x+t)-f(x)\sim-a_0t/x^2+\cdots$. The second difference has leading term
$-2a_0\b/x^3=2a_0\delta^2/x^3$, which agrees with
$\delta^2f''(x+\a)$ to leading order and displays the discrete convexity test.

\section{The Mills ratio and Hermite polynomials}\label{sec:mills}
Let $\varphi,\Phi$ be the standard normal density and distribution and
\[
M(z)=\frac{1-\Phi(z)}{\varphi(z)}=e^{z^2/2}\!\int_z^\infty e^{-u^2/2}\,\dd u
\]
the Mills ratio. Its classical divergent expansion is
\[
M(x)\sim\sum_{k\ge0}(-1)^k(2k-1)!!\,x^{-2k-1},
\qquad (-1)!!=1.
\]
Since the probabilists' Hermite
polynomials have $\He_{2k}(0)=(-1)^k(2k-1)!!$ and $\He_{2k+1}(0)=0$, this reads
$M(x)\sim\sum_n\He_n(0)x^{-n-1}$. Under a shift, the coefficients are the full
polynomials $\He_n(t)$.

\begin{lemma}[Laplace representation]\label{lem:lap}
For every $z\in\R$,
\[
M(z)=\int_0^\infty e^{-zs}e^{-s^2/2}\,\dd s.
\]
\end{lemma}
\begin{proof}
Substitute $u=z+s$ in the defining integral. Then
\[
M(z)=e^{z^2/2}\int_0^\infty e^{-(z+s)^2/2}\,\dd s
=\int_0^\infty e^{-zs-s^2/2}\,\dd s.
\]
The integral converges for all real $z$ because of the Gaussian factor.
\end{proof}

This representation is classical; Luci\'c \cite{Lucic} attributes it to Jacobi and uses
it, together with the Hermite expansion of monomials, for the derivatives $M^{(n)}$.
Our use of it is under a shift, in Theorem~\ref{thm:millsher}.

\begin{theorem}[Mills--Hermite, with remainder]\label{thm:millsher}
For every $x>0$, $t\in\R$ and integer $N\ge1$,
\begin{equation}\label{eq:millsher}
M(x+t)=\sum_{n=0}^{N-1}(-1)^n\frac{\He_n(t)}{x^{n+1}}+\rho_N(x,t),
\end{equation}
where
\begin{equation}\label{eq:rembound}
|\rho_N(x,t)|\le\frac{M_N(t)}{x^{N+1}},\qquad
M_N(t):=\sup_{s\ge0}\bigl|\He_N(s+t)\bigr|e^{-ts-s^2/2}<\infty .
\end{equation}
In particular, \eqref{eq:millsher} is an asymptotic expansion as $x\to\infty$,
uniformly for $t$ in every compact interval.
\end{theorem}
\begin{proof}
By Lemma~\ref{lem:lap}, $M(x+t)=\int_0^\infty e^{-xs}g(s)\dd s$ with
$g(s)=e^{-ts-s^2/2}$. The Hermite generating function $e^{ys-s^2/2}=\sum_n\He_n(y)s^n/n!$
at $y=-t$, together with the parity $\He_n(-t)=(-1)^n\He_n(t)$, gives
$g(s)=\sum_n(-1)^n\He_n(t)s^n/n!$. Write the Taylor expansion with integral remainder,
\[
\begin{aligned}
g(s)&=\sum_{n=0}^{N-1}(-1)^n\He_n(t)\frac{s^n}{n!}+r_N(s),\\
r_N(s)&=\frac1{(N-1)!}\int_0^s(s-u)^{N-1}g^{(N)}(u)\,\dd u .
\end{aligned}
\]
Because
\[
\frac{\dd^N}{\dd z^N}e^{-z^2/2}=(-1)^N\He_N(z)e^{-z^2/2},
\qquad g(s)=e^{t^2/2}e^{-(s+t)^2/2},
\]
we have
\[
g^{(N)}(s)=(-1)^N\He_N(s+t)e^{-ts-s^2/2}.
\]
Hence $|g^{(N)}(s)|\le M_N(t)$ and
$|r_N(s)|\le M_N(t)s^N/N!$ for $s\ge0$. Multiplying by $e^{-xs}$, integrating,
and using $\int_0^\infty s^Ne^{-xs}\dd s=N!/x^{N+1}$,
\[
|\rho_N(x,t)|=\Bigl|\int_0^\infty e^{-xs}r_N(s)\,\dd s\Bigr|
\le\frac{M_N(t)}{N!}\int_0^\infty s^Ne^{-xs}\dd s=\frac{M_N(t)}{x^{N+1}}.
\]
Finally,
$|\He_N(s+t)|e^{-ts-s^2/2}
=e^{t^2/2}|\He_N(s+t)|e^{-(s+t)^2/2}\to0$
as $s\to\infty$. Continuity gives $M_N(t)<\infty$, and the same argument on a
compact set of values of $t$ gives the stated local uniformity.
\end{proof}

The coefficients $\He_n(t)$ have no fixed sign pattern in $n$, so there is no general
alternating-series rule for the Hermite partial sums. Guaranteed two-sided bounds come
instead from the \emph{underlying} Mills ratio at $z=x+t$, where the scalar series
alternates.

\begin{proposition}[Two-sided envelope]\label{prop:env}
For $z=x+t>0$ and every $K\ge1$,
\begin{equation}\label{eq:env}
\begin{aligned}
M(x+t)&=\sum_{k=0}^{K-1}
\frac{(-1)^k(2k-1)!!}{(x+t)^{2k+1}}+r_K,\\
\sgn r_K&=(-1)^K,\qquad
|r_K|<\frac{(2K-1)!!}{(x+t)^{2K+1}},
\end{aligned}
\end{equation}
so the even and odd partial sums bracket $M(x+t)$ from below and above; the optimal
truncation by least term occurs near $K^*\approx(x+t)^2/2$.
\end{proposition}
\begin{proof}
Repeated integration by parts gives
\[
r_K=(-1)^K(2K-1)!!\,e^{z^2/2}
\int_z^\infty u^{-2K}e^{-u^2/2}\,\dd u.
\]
The integral is
positive, which determines the sign. Since $u^{-2K}\le z^{-2K}$ on $[z,\infty)$,
\[
\int_z^\infty u^{-2K}e^{-u^2/2}\dd u\le z^{-2K}\int_z^\infty e^{-u^2/2}\dd u
<z^{-2K-1}e^{-z^2/2},
\]
the last step by the bound $\int_z^\infty e^{-u^2/2}\dd u<z^{-1}e^{-z^2/2}$. This
proves the remainder bound.
\end{proof}

The two expansions organise the same function in different ways. Formula
\eqref{eq:env} uses the variable $x+t$ and the sign-definite numbers
$(2k-1)!!=(-1)^k\He_{2k}(0)$, and provides rigorous bounds for computation.
Formula \eqref{eq:millsher} uses the scale $x$ and displays the polynomial dependence
on the shift, as required for the formulas of \S\ref{sec:framework}. As an
illustration of the first difference \eqref{eq:dd}, the
increment from the base point is
\[
M(x+t)-M(x)\sim\sum_{n\ge1}(-1)^n\frac{\He_n(t)-\He_n(0)}{x^{n+1}}
=-\frac t{x^2}+\frac{t^2}{x^3}-\frac{t^3-3t}{x^4}+\cdots,
\]
one order smaller than $M\sim x^{-1}$ (the $1/x$ term cancels), with leading
$-t/x^2\sim tM'(x)$. Thus \eqref{eq:millsher} is adapted to symbolic dependence on
the shift, while \eqref{eq:env} is adapted to certified numerical evaluation.

\section{Applications: the non-central tail and the hazard}\label{sec:apps}
We now apply the preceding formulas with $x$ as the asymptotic variable and with
the shift fixed. All coefficients are then given explicitly in terms of Hermite
polynomials.

\subsection{The non-central Gaussian tail}
Let $X\sim N(\mu,1)$. Then
$P(X>a)=1-\Phi(a-\mu)=\varphi(a-\mu)M(a-\mu)$. For fixed $\mu$, two useful forms as
$a\to\infty$ follow from \S\ref{sec:mills}.

\begin{theorem}[Non-central tail]\label{thm:noncentral}
\textup{(i)} For $a>\mu$ and every $K$,
\begin{equation}\label{eq:alt}
\begin{aligned}
P(X>a)
&=\varphi(a-\mu)\Bigl[
\sum_{k=0}^{K-1}\frac{(-1)^k(2k-1)!!}{(a-\mu)^{2k+1}}+r_K\Bigr],\\
\sgn r_K&=(-1)^K,\qquad
|r_K|<\frac{(2K-1)!!}{(a-\mu)^{2K+1}},
\end{aligned}
\end{equation}
so the partial sums bracket the tail. \textup{(ii)} As $a\to\infty$, locally uniformly
for $\mu$ in compact sets, with $\varphi(a-\mu)=\varphi(a)e^{a\mu-\mu^2/2}$,
\begin{equation}\label{eq:noncentral}
P(X>a)\sim\varphi(a)\,e^{a\mu-\mu^2/2}\sum_{n=0}^\infty\frac{\He_n(\mu)}{a^{n+1}}.
\end{equation}
\end{theorem}
\begin{proof}
(i) is Proposition~\ref{prop:env} at $z=a-\mu$. For (ii), put $x=a$, $t=-\mu$ in
\eqref{eq:millsher} and use $\He_n(-\mu)=(-1)^n\He_n(\mu)$, so the signs cancel and
$M(a-\mu)\sim\sum_n\He_n(\mu)a^{-n-1}$; multiply by $\varphi(a-\mu)$.
\end{proof}

Although the limiting variable in \eqref{eq:noncentral} is $a$, its coefficients
display the dependence on the non-centrality $\mu$ explicitly --- the regime
$a\to\infty$ with $\mu$ held fixed (the threshold large relative to the mean). Read as
the power
$\pi(\mu):=P(X>a)$ of the one-sided test ``reject if $X>a$'', its value at the null is
the size, and the corresponding asymptotic expansion is
\begin{equation}\label{eq:size}
\pi(0)=\varphi(a)M(a)=1-\Phi(a)=\alpha,\qquad\text{with}\qquad
\pi(0)\sim\varphi(a)\sum_{n\ge0}\frac{\He_n(0)}{a^{n+1}}.
\end{equation}
Table~\ref{tab:noncentral} illustrates the accuracy of \eqref{eq:noncentral} for
several non-centralities.

\begin{table}[ht]\centering\small
\begin{tabular}{ccccc}
\toprule
$\mu$ & $P(X>a)$ exact & leading term & expansion \eqref{eq:noncentral}, $7$ terms & rel.\ error\\
\midrule
$0$   & $3.1671\cdot10^{-5}$ & $3.3458\cdot10^{-5}$ & $3.1636\cdot10^{-5}$ & $1.1\cdot10^{-3}$\\
$0.5$ & $2.3263\cdot10^{-4}$ & $2.1817\cdot10^{-4}$ & $2.3295\cdot10^{-4}$ & $1.4\cdot10^{-3}$\\
$1$   & $1.3499\cdot10^{-3}$ & $1.1080\cdot10^{-3}$ & $1.3525\cdot10^{-3}$ & $1.9\cdot10^{-3}$\\
$1.5$ & $6.2097\cdot10^{-3}$ & $4.3821\cdot10^{-3}$ & $6.2052\cdot10^{-3}$ & $7.2\cdot10^{-4}$\\
$2$   & $2.2750\cdot10^{-2}$ & $1.3498\cdot10^{-2}$ & $2.2662\cdot10^{-2}$ & $3.9\cdot10^{-3}$\\
\bottomrule
\end{tabular}
\caption{Non-central tail $P(X>a)$ for $X\sim N(\mu,1)$ at threshold $a=4$; size
$\alpha=1-\Phi(4)=3.17\cdot10^{-5}$ is the $\mu=0$ row. ``Leading'' is the $n=0$ term
$\varphi(a)e^{a\mu-\mu^2/2}/a$; the expansion is \eqref{eq:noncentral} truncated after
seven terms ($n\le6$). The Hermite series oscillates in $n$, so seven terms already give
three figures; rigorous brackets come from the alternating form \eqref{eq:alt}. Values
computed in \texttt{mpmath} at $40$-digit precision.}
\label{tab:noncentral}
\end{table}

The exact identity $\pi(\mu)=\Phi(\mu-a)$ gives
$\partial_\mu\pi=\varphi(a-\mu)$. The formal Hermite expansion is compatible with
this identity. Indeed, writing
$G(\mu)\sim\sum_{n\ge0}\He_n(\mu)a^{-n-1}$, the Appell property
$\He_n'=n\He_{n-1}$ and the Hermite recurrence give, as an identity of formal
series,
\[
(a-\mu)G+G'=\sum_n\frac{(a-\mu)\He_n(\mu)+n\He_{n-1}(\mu)}{a^{n+1}}=1,
\]
where the last equality telescopes coefficient by coefficient. Consequently the
formal derivative of \eqref{eq:noncentral} is $\varphi(a-\mu)$, as required. In
particular, the exact local power slope at the null is $\pi'(0)=\varphi(a)$.

\subsection{Sensitivity and curvature in the mean}
The reduced tail
$r(\mu):=M(a-\mu)\sim\sum_{n\ge0}\He_n(\mu)\,a^{-n-1}$ has Hermite polynomials
in the mean $\mu$. Therefore the finite-difference formulas of
Proposition~\ref{prop:diff} apply in $\mu$ (the substitution $t=-\mu$ and the Hermite
parity make the expansions non-alternating). We evaluate the reduced tail at
$\mu$ and $\mu\pm\delta$, which represents either a sensitivity calculation or a
symmetric two-point uncertainty in the mean. With
$\b=-\delta^2$,
\begin{equation}\label{eq:mean-first}
\begin{aligned}
\frac{r(\mu+\delta)-r(\mu-\delta)}{2\delta}
&\sim\sum_{m\ge0}\frac{(m+1)\,\nab_m(\mu,\b)}{a^{m+2}},\\
\nab_m(\mu,\b)
&=\sum_{i=0}^{\lfloor m/2\rfloor}
\frac1{2i+1}\binom m{2i}(-\b)^i\He_{m-2i}(\mu),
\end{aligned}
\end{equation}
a central (two-point) difference estimate of the sensitivity $\partial_\mu r=-M'(a-\mu)$, with leading term
$a^{-2}$ and the $O(\delta^2)$ bias displayed term by term. The symmetric second
difference is
\begin{equation}\label{eq:mean-second}
\begin{aligned}
r(\mu+\delta)-2r(\mu)+r(\mu-\delta)
&\sim\sum_{n\ge2}\frac{\Delta_n(\mu,\b)}{a^{n+1}},\\
\Delta_n(\mu,\b)
&=2\sum_{i=1}^{\lfloor n/2\rfloor}
\binom n{2i}(-\b)^i\He_{n-2i}(\mu),
\end{aligned}
\end{equation}
with leading term $2\delta^2/a^3>0$. Since $M$ is convex, $r$ is convex in $\mu$, so a
symmetric spread $\mu\pm\delta$ of the mean raises the average,
$\tfrac12[r(\mu+\delta)+r(\mu-\delta)]\ge r(\mu)$, and \eqref{eq:mean-second} is its
asymptotic expansion. The same sign holds for the exceedance probability itself when the
spread stays below the threshold: $P(X>a)=\Phi(\mu-a)$ has
$\partial_\mu^2P=(a-\mu)\varphi(a-\mu)>0$ for $\mu<a$, so for $\mu+\delta<a$ averaging
over a symmetric uncertainty in the mean can only increase the expected tail
probability.

\subsection{Mean-excess and hazard}
The hazard rate (reciprocal Mills) $\lambda=1/M=\varphi/(1-\Phi)$ is obtained by
inverting the Mills expansion. Factoring out the leading term,
$M(x+t)=x^{-1}\bigl(1+O(x^{-1})\bigr)$, the normalised series $xM(x+t)\sim1-t/x+\cdots$
has leading coefficient $1$, so its reciprocal is again a well-defined asymptotic series
\cite{Olver}. Multiplying back by $x$, with the leading $x+t$ exact and the corrections
organised in the scale $x$,
\[
\lambda(x+t)\sim (x+t)+\frac1x-\frac t{x^2}+\frac{t^2-2}{x^3}+\cdots .
\]
The mean residual life (or mean excess) is
$e(x)=\E[X-x\mid X>x]=\dfrac{\varphi(x)}{1-\Phi(x)}-x=\lambda(x)-x$, not the Mills ratio
$M(x)$, so
$e(x+t)=\lambda(x+t)-(x+t)\sim\dfrac1x-\dfrac t{x^2}+\dfrac{t^2-2}{x^3}+\cdots$. The
inverse Mills ratio $\lambda$ is the selection-correction term of the Heckman model
\cite{Heckman}, so the expansion describes that correction at a large argument; any
connection to extreme-value mean-excess analysis we leave as a potential application
rather than develop here.

\section{Stieltjes-admissible sequences and two-sided bounds}\label{sec:stieltjes}
We next consider a positive measure $\mu$ on $[0,\infty)$ with finite moments of
every order:
\begin{equation}\label{eq:moment-kernel}
a_n=\int_0^\infty\sigma^n\,\dd\mu(\sigma),\qquad
A(-s)=\int_0^\infty e^{-\sigma s}\,\dd\mu(\sigma).
\end{equation}
Here the second formula defines the kernel on $s\ge0$; its right derivatives satisfy
$\frac{\dd^n}{\dd s^n}A(-s)|_{s=0}=(-1)^na_n$, even when the formal series
$A(z)=\sum a_nz^n/n!$ has zero radius of convergence.
Thus $(a_n)$ is a Stieltjes moment sequence and $A(-s)$ is completely monotone.
Equivalently, for every size, both Hankel matrices $[a_{i+j}]$ and
$[a_{i+j+1}]$ are positive semidefinite \cite{Akhiezer,Widder}. Fubini's theorem gives
\begin{equation}\label{eq:stieltjes-transform}
f(x)=\int_0^\infty e^{-xs}A(-s)\,\dd s
=\int_0^\infty\frac{\dd\mu(\sigma)}{x+\sigma}
\sim\sum_{n\ge0}(-1)^n\frac{a_n}{x^{n+1}},\qquad x>0.
\end{equation}
We call such a sequence \emph{Stieltjes-admissible}.

If $t+\operatorname{supp}\mu\subset[0,\infty)$, in particular if $t\ge0$, then
$f(x+t)$ is again a Stieltjes transform in the variable $x$, with shifted moments
$R_n(t)=\int(t+\sigma)^n\,\dd\mu(\sigma)$, and all bounds below remain valid after
replacing $a_n$ by $R_n(t)$. For $t<0$ the shifted sequence may cease to be a Stieltjes
moment sequence; the shift is thus one-sided here, in contrast to the Gaussian regime,
where $t$ is unrestricted.

\begin{theorem}[The bounding staircase]\label{thm:staircase}
Let $(a_n)$ be Stieltjes-admissible and $a_0>0$. Put
\[
F(z):=\frac{f(1/z)}{z}=\int_0^\infty\frac{\dd\mu(\sigma)}{1+\sigma z}
\sim\sum_{n\ge0}(-1)^na_nz^n,
\]
and let $[p/q]_F$ denote its Pad\'e approximant at $z=0$. For $z>0$ and every level $m$
at which the approximants are nondegenerate,
\begin{equation}\label{eq:pade-staircase}
[m-1/m]_F(z)\le F(z)\le[m/m]_F(z),
\end{equation}
with the monotone ordering
\[
[0/1]_F\le[1/2]_F\le\cdots\le F
\le\cdots\le[2/2]_F\le[1/1]_F.
\]
All levels are nondegenerate when $\operatorname{supp}\mu$ is infinite; if $\mu$ has
exactly $N$ support points, the staircase terminates at level $N$, where
$F=[N-1/N]_F=[N/N]_F$ exactly (equivalently $f(x)=z\,[N/N]_F(z)$ at $z=1/x$), and the
higher approximants are read in their reduced (defective) form. The bracket at level $m$
uses $a_0,\dots,a_{2m}$; after multiplication by $z=1/x$ its width is $O(x^{-2m-1})$. At
the first level, if $a_1>0$,
\begin{equation}\label{eq:bound}
\frac{a_0}{x+a_1/a_0}\le f(x)\le
\frac{a_0x+(a_0a_2-a_1^2)/a_1}{x^2+(a_2/a_1)x}.
\end{equation}
The inequalities are strict unless the corresponding quadrature is exact; in particular,
for a point mass at $\sigma_0>0$ both are equalities, while for a point mass at the
origin $a_1=0$ and only the lower bound applies (the upper formula assumes $a_1>0$).
\end{theorem}
\begin{proof}
The Pad\'e bracketing and monotonicity in \eqref{eq:pade-staircase} are classical
properties of Stieltjes series \cite{Baker,Akhiezer}. The two approximants agree
through degree $2m-1$ in $z$, which gives the stated order of their difference
after multiplication by $z$.

For the first lower bound, normalise $\mu$ to a probability measure and apply
Jensen's inequality to the convex function $\sigma\mapsto(x+\sigma)^{-1}$:
\[
f(x)\ge\frac{a_0}{x+a_1/a_0},
\]
which is $z[0/1]_F(z)$ at $z=1/x$. The upper bound is $z[1/1]_F(z)$ at $z=1/x$; matching
the coefficients $a_0,a_1,a_2$ gives exactly the rational function on the right of
\eqref{eq:bound}. Both sides of the bracket are thus Pad\'e convergents.
\end{proof}

The same moment functional generates the formal $J$-fraction
\begin{equation}\label{eq:Jfraction}
\cfrac{a_0}{x+\b_0-\cfrac{\g_1}{x+\b_1-\cfrac{\g_2}{x+\b_2-\cdots}}},
\qquad \b_n\ge0,\quad \g_n>0,
\end{equation}
whose successive convergents are the \emph{subdiagonal} sequence $z\,[m-1/m]_F(z)$
($z=1/x$) --- the lower members of \eqref{eq:pade-staircase}. The companion upper,
diagonal sequence $z\,[m/m]_F(z)$ is furnished by the associated Gauss--Radau
approximants (equivalently, by the corresponding Stieltjes $S$-fraction), not by the
convergents of \eqref{eq:Jfraction}. Here $\g_n>0$ expresses the positive-definiteness of
the moment functional and $\b_n\ge0$ the confinement of $\operatorname{supp}\mu$ to
$[0,\infty)$ \cite{Akhiezer,Wall}; for a measure with finite support of $N$ points the
fraction terminates ($\g_n>0$ only for $n<N$), so $\g_n>0$ for all $n$ requires infinite
support. The two sequences converge to a common limit, which is then $f$, precisely
when the Stieltjes moment problem is \emph{determinate}; in this case the moments
determine $\mu$. This holds automatically for compactly supported $\mu$, and under
Carleman's condition
$\sum_n a_n^{-1/(2n)}=\infty$ in general \cite{Akhiezer}. When it is \emph{indeterminate},
distinct measures may share every moment, hence every $\g_n,\b_n$, while their transforms
differ; the bracketing convergents of \eqref{eq:pade-staircase} still trap each such $f$
but no longer single out the measure. The three examples below are all determinate
(compact support for Marchenko--Pastur; Carleman's condition for the exponential and
Poisson laws).

\subsection{Three examples with the same first moments}\label{sec:three-moments}
Because level $m$ uses only the moments up to $a_{2m}$, two measures that agree in
their first moments have the same initial bounds. The following three probability
measures all satisfy
\[
(a_0,a_1,a_2)=(1,1,2).
\]
Consequently, Theorem~\ref{thm:staircase} gives the same first-level bracket for all
three associated Stieltjes transforms:
\begin{equation}\label{eq:common-first-bound}
\frac1{x+1}\le f(x)\le\frac{x+1}{x^2+2x},\qquad x>0.
\end{equation}
The next two moments distinguish the examples.

\paragraph{Example 1: factorial moments.}
Let $S$ have the exponential distribution with mean one,
$\dd\mu(\sigma)=e^{-\sigma}\dd\sigma$ \cite{JohnsonKotzBalakrishnan}. Then
\[
a_n=\E[S^n]=n!,\qquad
(a_0,\dots,a_4)=(1,1,2,6,24),
\]
and the corresponding Stieltjes transform is
\[
f_{\rm Exp}(x)=\int_0^\infty\frac{e^{-\sigma}}{x+\sigma}\,\dd\sigma
=e^xE_1(x),\qquad
E_1(x):=\int_x^\infty\frac{e^{-u}}u\,\dd u.
\]
The level-two Pad\'e bounds are
\begin{equation}\label{eq:exp-level-two}
\frac{x+3}{x^2+4x+2}
\le e^xE_1(x)\le
\frac{x^2+5x+2}{x(x^2+6x+6)}.
\end{equation}
For instance, at $x=2$,
\[
\frac5{14}=0.35714
<e^2E_1(2)=0.36133
<\frac4{11}=0.36364.
\]

\paragraph{Example 2: Bell moments.}
Let $N\sim\operatorname{Poisson}(1)$. Its moments are the Bell numbers $B_n$:
\[
a_n=\E[N^n]=B_n,\qquad
(a_0,\dots,a_4)=(1,1,2,5,15).
\]
Indeed, Dobi\'nski's formula
$B_n=e^{-1}\sum_{k\ge0}k^n/k!$ identifies the representing measure explicitly
\cite{Rota}. Equivalently,
$\mu=e^{-1}\sum_{k\ge0}\delta_k/k!$, and therefore
\[
f_{\rm Bell}(x)
=e^{-1}\sum_{k=0}^\infty\frac1{k!(x+k)}.
\]
The level-two bounds become
\begin{equation}\label{eq:bell-level-two}
\frac{x+2}{x^2+3x+1}
\le f_{\rm Bell}(x)\le
\frac{x^2+4x+2}{x(x^2+5x+5)}.
\end{equation}
At $x=2$ the middle value simplifies to $f_{\rm Bell}(2)=e^{-1}$, so
\[
\frac4{11}=0.36364
<e^{-1}=0.36788
<\frac7{19}=0.36842.
\]
This example also makes admissibility immediate: the Bell numbers are moments of
a positive measure, rather than merely a combinatorial sequence.

\paragraph{Example 3: Catalan moments.}
Let
\[
C_n=\frac1{n+1}\binom{2n}{n}
\]
be the Catalan numbers. They are the moments of the Marchenko--Pastur probability
measure with parameter one on $[0,4]$ \cite{MarchenkoPastur,NicaSpeicher},
\[
\dd\mu(\sigma)=\frac1{2\pi}\sqrt{\frac{4-\sigma}{\sigma}}\,
\mathbf1_{(0,4)}(\sigma)\,\dd\sigma.
\]
Thus
\[
(a_0,\dots,a_4)=(1,1,2,5,14),
\qquad
f_{\rm Cat}(x)
=\int_0^4\frac{\dd\mu(\sigma)}{x+\sigma}
=\frac{\sqrt{1+4/x}-1}{2}.
\]
Its level-two bounds are
\begin{equation}\label{eq:catalan-level-two}
\frac{x+2}{x^2+3x+1}
\le f_{\rm Cat}(x)\le
\frac{x^2+3x+1}{x(x+1)(x+3)}.
\end{equation}
At $x=2$ this reads
\[
\frac4{11}=0.36364
<\frac{\sqrt3-1}{2}=0.36603
<\frac{11}{30}=0.36667.
\]

The comparison shows how the staircase separates the three transforms moment by moment.
The common moments $a_0,a_1,a_2$ give the common bracket
\eqref{eq:common-first-bound}. The factorial sequence has $a_3=6$, whereas the
Bell and Catalan sequences have $a_3=5$; hence its lower level-two bound is
different. Bell and Catalan still share that lower bound, but their values
$a_4=15$ and $a_4=14$ produce different upper level-two bounds. Since Bell and Catalan
agree through $a_3$, their transforms first separate at order $x^{-5}$, the width of the
level-two bracket.

\section{Concluding remarks}
The shift theorem can be extended in two directions. The Appell sequence may be
changed, as in \S\S\ref{sec:mills}--\ref{sec:stieltjes}, or the sequence may be kept
fixed while the shift depends on the asymptotic variable, as in
Corollary~\ref{cor:moving}. The second case explains the occurrence of Bernoulli
polynomials at a non-convergent shift in the asymptotic expansion of De~Moivre's
mean absolute deviation \cite{ElezovicMAD}. A combination of a non-Bernoulli sequence
with a moving shift gives a further possible direction.

The Borel--Laplace representation \eqref{eq:laplace} provides a common analytic
framework for two developments beyond gamma-related functions. For the Gaussian
kernel, it gives the shifted Mills expansion with an explicit remainder estimate.
The Hermite form of the coefficients leads to finite-difference formulas and to the
large-threshold expansion of a non-central Gaussian tail. For Stieltjes kernels, the
same generating sequence is a moment sequence, and its moments determine successive
Pad\'e bounds. In this way, the exponential generating function $A$, its Borel kernel
$A(-s)$ and the Appell polynomials of the shift connect the two classes.

The formal Appell construction applies to every coefficient sequence. Its Laplace
representation requires suitable growth and convergence conditions. In the
Stieltjes case, determinacy is additionally needed if the measure is to be recovered
from its moments. These conditions separate the formal and analytic parts of the
method.

\bigskip

\noindent\textbf{Author contributions} 
All authors have contributed equally to this manuscript.\\

\noindent\textbf{Data Availability Statement} 
No datasets were generated or analyzed during
the current study.\\

\noindent\textbf{Declarations}\\

\noindent\textbf{Conflict of interest} 
The authors declare no conflict of interest.\\

\noindent\textbf{Acknowledgements} 
This work was supported by the Croatian Science Foundation (HRZZ) under the project UIP-2025-02-8956.

\end{document}